\documentclass[11pt,a4paper]{amsart}
\usepackage[T1]{fontenc}
\usepackage{lmodern}
\usepackage{mathtools,amssymb,amsthm,enumitem,setspace,xcolor}
\usepackage[top=3.2cm,bottom=3.8cm,left=3cm,right=2cm]{geometry}
\usepackage[unicode=true,colorlinks,linkcolor=blue,citecolor=blue,urlcolor=blue,pagebackref]{hyperref}
\allowdisplaybreaks[1]
\mathtoolsset{showonlyrefs}
\newcommand{\papersection}[1]{\par\vskip25pt\section{\normalfont\bfseries #1}\vskip5pt}
\newcommand{\papersubsection}[1]{\subsection{#1}\ \par\nopagebreak[4]}
\makeatletter
\renewcommand{\@defaultbiblabelstyle}[1]{[#1]}
\makeatother
\AtBeginDocument{\typeout{PETTY-LAYOUT width=\the\textwidth; height=\the\textheight; baseline=\the\baselineskip; left=\the\oddsidemargin; top=\the\topmargin; headheight=\the\headheight; headsep=\the\headsep; footskip=\the\footskip}}
\hypersetup{pdftitle={Petty's equilateral-set conjecture in dimension five},pdfauthor={Kaiwen Yang}}
\newtheorem{theorem}{Theorem}[section]
\newtheorem{lemma}[theorem]{Lemma}

\theoremstyle{remark}
\newtheorem{remark}[theorem]{Remark}
\numberwithin{equation}{section}
\newcommand{\R}{\mathbb R}
\newcommand{\F}{\mathbb F}
\newcommand{\B}{\mathcal B}
\newcommand{\rank}{\operatorname{rank}}
\newcommand{\osc}{\operatorname{osc}}
\newcommand{\Hom}{\operatorname{Hom}}
\newcommand{\SO}{\operatorname{SO}}

\newcommand{\bal}{\mathrm{bal}}
\begin{document}
\enlargethispage{3pt}
\begin{center}
{\Large\bfseries Petty's equilateral-set conjecture in dimension five}
\vskip12pt
{\large Kaiwen Yang\textsuperscript{1,2}}\\[5pt]
{\small\textsuperscript{1}School of Mathematical Sciences, Tongji University\\
\textsuperscript{2}Tencent Hunyuan\\[3pt]
Email: \href{mailto:yangkaiwen@tongji.edu.cn}{\nolinkurl{yangkaiwen@tongji.edu.cn}}}
\end{center}
\vskip20pt
\begin{center}
\begin{minipage}{14cm}
\textbf{Abstract.} Petty conjectured that every $n$-dimensional real normed space contains
$n+1$ equidistant points. The conjecture was proved in dimensions three
and four by Petty and Makeev, respectively. We prove the five-dimensional
case: every five-dimensional real normed space contains six equidistant
points. The proof uses a bordism of equilateral configurations and a
homological obstruction.

\vskip5pt
\textbf{AI declaration.} The proof was generated by Tencent Hunyuan's
AI agent hyprover on September~29, 2026. The author organized, verified,
and rewrote the proof. The prompt was written by the author and is
reproduced in Appendix~\ref{app:prompt}.


\end{minipage}
\end{center}

\papersection{Introduction}

A finite set $S$ in a real normed space $X$ is \emph{equilateral} if
there is a number $\lambda>0$ such that
\[
 \|x-y\|=\lambda,\qquad x,y\in S,\quad x\ne y.
\]
Let $e(X)$ denote the largest cardinality of an equilateral subset of
$X$. Petty's equilateral-set conjecture~\cite{Petty1971} asserts that
\[
 e(X)\geq \dim X+1
\]
for every finite-dimensional real normed space. The Euclidean regular
simplex shows that this proposed universal lower bound is sharp.
Equilateral sets are antipodal~\cite{Petty1971}. The theorem of Danzer
and Gr\"unbaum~\cite{DanzerGrunbaum1962} therefore yields
$e(X)\leq 2^{\dim X}$. For background on equilateral sets
and their relation to Minkowski geometry, see~\cite{Swanepoel2004,Thompson1996}.

The conjecture is immediate in dimension one. In the plane, choose a
unit vector $u$ and a continuous arc $\gamma$ on the unit circle from
$u$ to $-u$. The function $t\mapsto\|\gamma(t)-u\|$ takes the values
$0$ and $2$, and hence takes the value $1$. This gives the equilateral
triangle $\{0,u,\gamma(t)\}$. Petty~\cite{Petty1971} proved the three-dimensional case, and
Kobos~\cite{Kobos2013} gave an alternative proof. Makeev~\cite{Makeev2005}
proved the four-dimensional case. The present paper establishes the next case.

\begin{theorem}\label{thm:main}
Every five-dimensional real normed space contains an equilateral set of
six points. Equivalently, for every norm $N$ on $\R^5$ there are
$x_0,\ldots,x_5\in\R^5$ such that $x_0=0$ and
\[
 N(x_i-x_j)=1,\qquad 0\leq i<j\leq5.
\]
\end{theorem}

The theorem concerns the existence of an equilateral set; it does not
require a prescribed set to extend. This distinction is relevant because
maximal equilateral sets can be small \cite{Petty1971,SwanepoelVilla2013}. Nor do we require the six points to
be affinely independent. Configurations of affine rank four occur in
the argument and must be retained.

We argue by contradiction. The absence of a six-point equilateral set
persists under sufficiently small perturbations of a norm. A hypothetical
counterexample can therefore be chosen smooth and strongly convex. Join
it to the Euclidean norm and perturb the path, keeping the endpoints
fixed. All centered, labeled unit-equilateral configurations along the
path form a compact oriented eleven-manifold whose boundary is the
Euclidean configuration space $O(5)$.

The main issue is the behavior of the configuration matrix at rank loss.
The classification of Sch\"urmann and Swanepoel \cite[Theorem~2.3]{SchurmannSwanepoel2006} excludes six equilateral points
in a strictly convex three-dimensional normed space. Thus every matrix
on the bordism has rank at least four. At rank four its kernel is a
line of affine relations. Strict antipodality gives an explicit
inequality for this line. Compactness places the relevant lines in a
region with a uniform balance margin, for which we construct an
equivariant graph model.

To finish, we cut the bordism at a small positive determinant level.
Positive polar decomposition maps the resulting quotient bordism to
$\SO(5)/C_5$, and is the identity on the Euclidean boundary. The other
boundary factors through a bundle over a graph with fiber $\SO(5)$.
Since $H_9(\SO(5);\F_5)=0$, its tenth homology is generated by fiber
classes. Each fiber maps to $\SO(5)/C_5$ with degree five, so this
boundary has zero image over $\F_5$. It cannot cancel the nonzero
fundamental class of the Euclidean boundary.

After the public release of this article, Logan Chalmers brought
to the author's attention his manuscript, posted on ResearchGate
on October~26, 2026~\cite{Chalmers2026Manuscript}, and his completed
Lean formalisation of the result, available on
Zenodo~\cite{Chalmers2026Lean}. The author had been unaware of
Chalmers's work. The proof presented here was generated independently
by hyprover on September~29, 2026, before Chalmers's manuscript
became publicly available. The author makes no claim to publication
priority for this result.

Section~\ref{sec:preliminaries} recalls the transversality and homology tools.
Section~\ref{sec:configurations} carries out the smooth reduction, establishes
the constraints on six-point configurations, and constructs the bordism
and its relation-line graph. Section~\ref{sec:proof} computes the required
homology, proves the determinant-cut obstruction, and concludes the
proof of Theorem~\ref{thm:main}.

\papersection{Preliminaries}\label{sec:preliminaries}

All vector spaces are real and finite-dimensional. We fix auxiliary
Euclidean inner products and, whenever determinants are used, orientations.
The notation $|\cdot|$, orthogonal complements and adjoints refer to
these inner products. A norm is \emph{smooth} if it is $C^\infty$
away from zero, \emph{strictly convex} if its unit sphere contains no
nontrivial line segment, and \emph{strongly convex} if
$D^2(N^2)(u)$ is positive definite for $u\ne0$.
Strong convexity implies strict convexity: on a segment avoiding zero,
$N^2$ is strictly convex, and a segment through zero cannot lie in the
unit sphere.

\papersubsection{Regular fibers and transversality}

We recall the regular-value and transversality results used to construct
the configuration bordism, together with the orientation convention
needed to pass to its quotient.

For a smooth map $F:M\to W$ to a vector space, transversality to zero
means that $DF_x:T_xM\to W$ is onto at every zero; the empty zero set
is allowed~\cite[p.~143]{Lee2013}.
We use the following forms of the regular-value theorem
\cite[Chapter~1]{GuilleminPollack1974} and its boundary version
\cite[Problem~5-23]{Lee2013}.

\begin{lemma}[Regular zero sets]\label{lem:regular-fiber}
Let $M$ be a smooth $m$-manifold, possibly with boundary, and let $W$
be an $r$-dimensional vector space. Suppose $F:M\to W$ is smooth and
transverse to zero. If $M$ has boundary, suppose also that
$F|_{\partial M}$ is transverse to zero. Then $Z=F^{-1}(0)$ is a
smooth submanifold of dimension $m-r$, possibly empty, with
\[
 T_xZ=\ker DF_x,\qquad \partial Z=Z\cap\partial M.
\]
The derivative gives the exact sequence of vector bundles
\[
 0\longrightarrow TZ\longrightarrow TM|_Z
 \xrightarrow{\ DF\ }Z\times W\longrightarrow0.
\]
If $M$ and $W$ are oriented, the identification
$\det TM|_Z\cong\det TZ\otimes\det W$ orients $Z$.
If a finite group acts smoothly on $M$ and linearly on $W$, preserves
both orientations, and $F$ is equivariant, then its action on $Z$
preserves this orientation.
\end{lemma}

The orientation is the preimage orientation~\cite[Chapter~15]{Lee2013}.
The equivariant assertion follows by taking determinant lines in the
exact sequence. Boundary orientations are taken with the outward
normal first.

Sard's theorem~\cite[Theorem~6.10]{Lee2013} supplies the regular values used below.

\begin{lemma}[Sard's theorem]\label{lem:sard}
Let $F:M\to N$ be a smooth map between finite-dimensional smooth
manifolds without boundary. The critical values of $F$ have measure
zero in every coordinate chart of $N$. Thus almost every $y\in N$ is
a regular value, meaning that $DF_x:T_xM\to T_yN$ is onto for every
$x\in F^{-1}(y)$. In particular, for a smooth real-valued function on
$M$, every nonempty open interval contains a regular value.
\end{lemma}

\begin{lemma}[Parameter transversality]\label{lem:parameter}
Let $M$ be a smooth manifold without boundary, $P\subset\R^L$ open,
$W$ a finite-dimensional vector space, and $\Phi:M\times P\to W$ smooth.
If $\Phi$ is transverse to zero,
then $\Phi_a(x)=\Phi(x,a)$ is transverse to zero for almost every
$a\in P$~\cite[Theorem~6.35 and errata to p.~146]{Lee2013}.
\end{lemma}

Smooth cutoffs will be used to keep endpoint neighborhoods fixed and
to localize perturbations. If $K$ is compact in a smooth manifold and
$K\subset U$ with $U$ open, there is a smooth function $\chi$ with
$0\leq\chi\leq1$, equal to one near $K$, whose support is a compact
subset of $U$~\cite[Chapter~2]{Lee2013}.
A \emph{bordism} between two closed smooth $k$-manifolds is a compact
smooth $(k+1)$-manifold whose boundary is diffeomorphic to their disjoint
union; here closed means compact without boundary.

\papersubsection{Homology of pairs, products and quotients}

We collect the exact sequences and product formulas used in the
homology calculation for bundles over the relation graph.

All homology groups have coefficients in $\F_5=\mathbb Z/5\mathbb Z$;
we suppress these coefficients in this subsection. We write $f_*$ for
the homomorphism induced by $f$, and $\widetilde H_*$ for reduced
homology. For a nonempty space, $\widetilde H_j=H_j$ for $j>0$ and
$\widetilde H_0$ is the kernel of the augmentation $H_0\to\F_5$.
The next lemma collects the forms of homotopy invariance, exactness
and excision that we use~\cite[Theorem~2.10 and Section~2.1]{Hatcher2002},
together with Mayer--Vietoris~\cite[Section~2.2, pp.~149--150]{Hatcher2002}.

\begin{lemma}\label{lem:homology-tools}
Homology with coefficients in $\F_5$ has the following homotopy,
exactness and excision properties.
\begin{enumerate}[label=\textup{(\roman*)}]
\item If $f,g:(X,A)\to(Y,B)$ are homotopic through maps of pairs,
then $f_*=g_*$ on relative homology. Consequently a homotopy
equivalence of pairs induces isomorphisms in all degrees. The same
statements hold for absolute homology by taking $A=B=\varnothing$.
\item For every pair $A\subset X$, there is a natural long exact sequence
\begin{equation}\label{eq:pair-sequence}
 \cdots\longrightarrow H_j(A)\xrightarrow{i_*}H_j(X)
 \xrightarrow{j_*}H_j(X,A)\xrightarrow{\partial}H_{j-1}(A)
 \longrightarrow\cdots.
\end{equation}
Here $i_*$ is induced by inclusion, $j_*$ passes to relative classes,
and $\partial$ takes a relative cycle to the class of its boundary
in $A$. In particular, if $H_j(X,A)=0$, then $H_j(A)\to H_j(X)$ is onto.
\item If $Z\subset A\subset X$ and $\overline Z\subset\operatorname{int}_X A$,
the inclusion of pairs induces an isomorphism
\[
 H_j(X\setminus Z,A\setminus Z)\xrightarrow{\ \cong\ }H_j(X,A)
\]
for every $j$.
\item If $X=U\cup V$ with $U,V$ open in $X$, there is a natural
Mayer--Vietoris exact sequence
\begin{equation}\label{eq:mv-sequence}
 \cdots\longrightarrow H_j(U\cap V)\xrightarrow{\alpha}
 H_j(U)\oplus H_j(V)\xrightarrow{\beta}H_j(X)
 \xrightarrow{\partial}H_{j-1}(U\cap V)\longrightarrow\cdots.
\end{equation}
With $i_U,i_V$ denoting inclusions from $U\cap V$ and $j_U,j_V$
inclusions into $X$, the maps are
\[
 \alpha(a)=\bigl((i_U)_*a,-(i_V)_*a\bigr),\qquad
 \beta(b,c)=(j_U)_*b+(j_V)_*c.
\]
\end{enumerate}
\end{lemma}

A pair $(X,A)$ is called \emph{good} if $A$ is a nonempty closed
subspace that is a deformation retract of an open neighborhood in $X$.
A CW pair is a CW complex together with a subcomplex.
The quotient and wedge formulas below hold for good pairs
\cite[Section~2.1, pp.~124--126]{Hatcher2002}; the product formulas
are the field-coefficient and relative K\"unneth formulas
\cite[Section~3.B, p.~276]{Hatcher2002}.

\begin{lemma}\label{lem:products-quotients}
The following quotient and product formulas hold for homology with
coefficients in $\F_5$.
\begin{enumerate}[label=\textup{(\roman*)}]
\item For a good pair $(X,A)$, the quotient map induces isomorphisms
\begin{equation}\label{eq:good-quotient}
 H_j(X,A)\cong\widetilde H_j(X/A),\qquad j\geq0.
\end{equation}
If finitely many based spaces $(X_\alpha,x_\alpha)$ are good pairs,
the inclusions into their wedge induce
\begin{equation}\label{eq:wedge-homology}
 \bigoplus_\alpha\widetilde H_j(X_\alpha)
 \cong\widetilde H_j\Bigl(\bigvee_\alpha X_\alpha\Bigr).
\end{equation}
\item For CW complexes $X,Y$, the cross product induces a natural
isomorphism
\begin{equation}\label{eq:kunneth}
 \bigoplus_{p+q=j}H_p(X)\otimes_{\F_5}H_q(Y)
 \xrightarrow{\ \cong\ }H_j(X\times Y).
\end{equation}
For CW pairs $(X,A)$ and $(Y,B)$, its relative form is
\begin{equation}\label{eq:relative-kunneth}
 \bigoplus_{p+q=j}H_p(X,A)\otimes_{\F_5}H_q(Y,B)
 \cong H_j\bigl(X\times Y,A\times Y\cup X\times B\bigr).
\end{equation}
These formulas also hold for spaces and pairs homotopy equivalent
to CW complexes and CW pairs, respectively.
\item In particular, if $X$ has CW homotopy type, $k\geq1$ and
$I=[0,1]$, then
\begin{align}
 H_j(X\times S^k)&\cong H_j(X)\oplus H_{j-k}(X),\label{eq:sphere-product}\\
 H_j\bigl(X\times I,X\times\partial I\bigr)&\cong H_{j-1}(X),
 \qquad j\geq1.\label{eq:interval-product}
\end{align}
Groups of negative degree are understood to be zero.
\end{enumerate}
\end{lemma}

The special cases~\eqref{eq:sphere-product} and~\eqref{eq:interval-product}
follow from~\eqref{eq:kunneth} and~\eqref{eq:relative-kunneth}.
They use $H_0(S^k)=H_k(S^k)=\F_5$ and zero in all other
degrees, together with $H_1(I,\partial I)=\F_5$ and zero in all
other degrees~\cite[Section~2.1]{Hatcher2002}.
Compact manifolds have CW homotopy type~\cite[Corollary~A.12]{Hatcher2002},
so these formulas apply to the compact Lie groups used below.

\papersubsection{Fundamental classes and finite coverings}

We record how fundamental classes behave under boundaries and finite
coverings. These identities provide the cancellation used in the
proof of the main theorem.

We restore the coefficient notation in the formulas below. A connected
closed oriented $d$-manifold has a fundamental class $[M]$, characterized
by the specified local orientation at each point, and
\[
 H_d(M;\F_5)=\F_5[M],\qquad H_j(M;\F_5)=0,\quad j>d
\]
by~\cite[Theorem~3.26]{Hatcher2002}. In particular, $[M]\ne0$.
For a compact oriented manifold $Y$ with boundary, the relative
fundamental class satisfies
\[
 \partial[Y,\partial Y]=[\partial Y]
\]
under the connecting homomorphism, with the induced boundary
orientation~\cite[Section~3.3, pp.~253--254]{Hatcher2002}.
For disconnected manifolds, these classes mean sums over their
oriented components.

\begin{lemma}\label{lem:boundary-cover}
Let $Y$ be a compact oriented $(d+1)$-manifold and $f:Y\to M$ continuous.
Then
\begin{equation}\label{eq:boundary}
 \bigl(f|_{\partial Y}\bigr)_*[\partial Y]=0\quad\text{in }H_d(M;\F_5).
\end{equation}
If $q:\widetilde M\to M$ is an orientation-preserving $r$-sheeted
covering of connected closed oriented $d$-manifolds, then
$q_*[\widetilde M]=r[M]$.
For a finite regular covering of path-connected, locally path-connected
spaces with deck group $G$,
there is a transfer
$\operatorname{tr}:H_j(M;\F_5)\to H_j(\widetilde M;\F_5)$ satisfying
\begin{equation}\label{eq:transfer}
 q_*\operatorname{tr}=|G|\,\mathrm{id},\qquad
 \operatorname{tr}q_*=\sum_{g\in G}g_*.
\end{equation}
If $|G|$ is invertible in $\F_5$ and every deck transformation acts
trivially on homology, then $q_*$ is an isomorphism with inverse
$|G|^{-1}\operatorname{tr}$.
\end{lemma}

\begin{proof}
By~\eqref{eq:pair-sequence}, the class $[\partial Y]$ vanishes in
$H_d(Y;\F_5)$; applying $f_*$ proves~\eqref{eq:boundary}.
An orientation-preserving $r$-sheeted covering has $r$ positive local
sheets, so its degree is $r$.
The transfer~\cite[Section~3.G]{Hatcher2002} sums all lifts of each
singular simplex and commutes with the boundary operator.
Projection gives $|G|$ copies of that simplex; in the other order,
the lifts are its deck translates. This proves~\eqref{eq:transfer},
and the last assertion follows by dividing both identities by $|G|$.
\end{proof}

\begin{lemma}[Finite quotients]\label{lem:finite-quotient}
Let a finite group $G$ act smoothly and freely on a smooth manifold
$Y$, possibly disconnected and with boundary. Then $Y/G$ has a smooth
manifold structure for which the quotient map is a $|G|$-sheeted
smooth covering, and
\[
 \partial(Y/G)=(\partial Y)/G.
\]
If $Y$ is oriented and the action preserves orientation, the quotient
has the unique orientation for which the covering is locally
orientation preserving. The displayed boundary identification then
preserves the induced orientations.
\end{lemma}

For connected manifolds without boundary this is the covering-space
quotient theorem~\cite[Theorem~21.13]{Lee2013}.
In general, choose coordinate neighborhoods disjoint from their
nontrivial translates, using half-space charts at boundary points.
Their projections give the quotient charts and the asserted boundary
identification; the orientations descend through the same charts.

\papersection{Equilateral configurations and relation lines}
\label{sec:configurations}

We first reduce to smooth strongly convex norms and establish the
geometric constraints on six-point configurations. These facts allow us
to construct a bordism along a path of norms and a graph model for its
relation lines, retaining the label action throughout.

\papersubsection{Reduction to smooth strongly convex norms}

We show that a counterexample, if one exists, can be chosen smooth and
strongly convex. We then establish the stability needed to perturb a
path of such norms.

We give the space of norms the topology of uniform convergence on a
fixed Euclidean unit sphere. Near each fixed norm $N$, this is equivalent
to multiplicative convergence. To see this,
write $x=\sum x_ie_i$; the triangle inequality gives
$N(x)\leq\sum |x_i|N(e_i)\leq C|x|$, so $N$ is continuous.
Its minimum $a$ on the Euclidean unit sphere is positive. If
$\sup_{|u|=1}|M(u)-N(u)|\leq\varepsilon a$, homogeneity gives
$(1-\varepsilon)N\leq M\leq(1+\varepsilon)N$. The converse follows
from the boundedness of $N$ on that sphere.

\begin{lemma}\label{lem:closed}
The norms on $\R^5$ admitting a six-point equilateral set form a closed
set.
\end{lemma}

\begin{proof}
Let $N_k\to N$, and translate and scale an $N_k$-equilateral set to
obtain $x_0^k=0$ and $N_k(x_i^k-x_j^k)=1$ for $i\ne j$.
There are constants $a,b>0$ such that
$a|v|\leq N_k(v)\leq b|v|$ for all sufficiently large $k$.
Consequently $|x_i^k|\leq a^{-1}$ and
$|x_i^k-x_j^k|\geq b^{-1}$. Passing to a subsequence gives limits
$x_i$ with $N(x_i-x_j)=1$. They are distinct and form the required
equilateral set.
\end{proof}

\begin{lemma}\label{lem:smoothing}
Every norm on $\R^5$ is a limit of smooth strongly convex norms.
Consequently, if Theorem~\ref{thm:main} fails, it fails for such a norm.
\end{lemma}

\begin{proof}
Fix a norm $N$ and $0<\eta<1$. A norming functional at a unit vector
$u$ is a linear form $f$ with $f(u)=1$ and $|f(x)|\leq N(x)$ for all
$x$. Such a form exists by Hahn--Banach, applied to $f(tu)=t$ on
$\R u$; see~\cite{Thompson1996}. Each form remains greater than
$1-\eta$ on a neighborhood of its contact point in the unit sphere.
Compactness therefore supplies finitely many forms $f_1,\ldots,f_L$ with
\[
 (1-\eta)N(x)\leq\max_j|f_j(x)|\leq N(x).
\]
Append sufficiently small nonzero multiples of all coordinate forms
and all forms $x_i+x_j$, keeping their dual norms at most one.
Write $L$ for the total number of forms after these additions.
For an even integer $p\geq4$ set
\[
 Q(x)=\sum_j f_j(x)^p,\qquad M(x)=Q(x)^{1/p}.
\]
The linear map $x\mapsto(f_j(x))_j$ is injective, and $M$ is the
pullback of the $\ell_p$ norm under this map. Thus it is a norm;
as $Q(x)>0$ for $x\ne0$, it is analytic away from zero. Moreover,
\[
 (1-\eta)N\leq M\leq L^{1/p}N.
\]
Choose $p$ so large that $L^{1/p}\leq1+\eta$.

For $x\ne0$,
\[
 D^2Q(x)[v,v]=p(p-1)\sum_j f_j(x)^{p-2}f_j(v)^2.
\]
If this vanishes, choose a coordinate $x_i\ne0$. The coordinate
forms first give $v_j=0$ whenever $x_j\ne0$; the form $x_i+x_j$
then gives $v_j=0$ for every remaining $j$. Thus $D^2Q(x)$ is
positive definite. To deduce strong convexity of $M$, write
$v=ax+w$ with $DM(x)[w]=0$. Homogeneity gives
\[
 D^2\bigl(M^2\bigr)(x)[v,v]
 =2a^2M(x)^2+2M(x)D^2M(x)[w,w].
\]
On $\ker DM(x)$, differentiation of $Q=M^p$ shows that
$D^2Q=pM^{p-1}D^2M$. The displayed expression is therefore positive
for $v\ne0$. Letting $\eta\to0$ proves the approximation claim.
The final assertion follows from Lemma~\ref{lem:closed}.
\end{proof}

For the construction below, approximation of a single norm is not enough.
We will perturb an entire path of squared norms by functions which need
not themselves be convex. The following observation explains why a
sufficiently small perturbation remains a path of norms.

\begin{lemma}\label{lem:perturbation}
Let $Q:[0,1]\times(\R^5\setminus\{0\})\to(0,\infty)$ be smooth,
even and homogeneous of degree two in the second variable. Assume
$D_u^2Q(t,u)$ is positive definite for every $t$ and $u\ne0$.
Then $N^t=\sqrt{Q(t,\cdot)}$ is a smooth strongly convex norm.
If $h_1,\ldots,h_L$ are smooth even degree-two homogeneous functions
away from zero, and $\chi$ is a smooth function of $t$, then the same
conclusion holds for
\[
 Q_{t,a}(u)=Q(t,u)+\chi(t)\sum_{j=1}^L a_jh_j(u)
\]
whenever the vector $a\in\R^L$ is sufficiently small. Such a parameter
ball can be chosen uniformly for all $t\in[0,1]$.
\end{lemma}

\begin{proof}
Extend each homogeneous function by zero at the origin. A smooth
degree-two homogeneous function has gradient homogeneous of degree one.
Bounds on the unit sphere show that this extension is $C^1$ at zero,
with derivative zero. The restriction of $Q(t,\cdot)$ to any affine
line has nonnegative second derivative away from a possible passage
through zero, and its first derivative is continuous at that passage.
It is therefore convex on the whole line. Thus
$K_t=\{u:Q(t,u)\leq1\}$ is convex and symmetric. Positivity and
compactness of the sphere make $K_t$ bounded with zero in its interior.
Its Minkowski functional, defined by
$p_{K_t}(u)=\inf\{r>0:u\in rK_t\}$, equals $\sqrt{Q(t,u)}$ by
homogeneity. Convexity gives
$p_{K_t}(u+v)\leq p_{K_t}(u)+p_{K_t}(v)$: divide $u+v$ by the sum
of the two gauges and use a convex combination of points in $K_t$.
Symmetry and boundedness give the other norm axioms. Its strong
convexity is the assumed Hessian condition.

On the compact set $[0,1]\times S^4$, choose $c,\lambda>0$ with
\[
 Q(t,u)\geq c,\qquad D_u^2Q(t,u)[v,v]\geq\lambda|v|^2.
\]
The values and Hessians of the finitely many $\chi(t)h_j$ are bounded
there. For $|a|$ sufficiently small their total contributions have
absolute value at most $c/2$ and Hessian operator norm at most
$\lambda/2$. This preserves positivity and positive-definite Hessians
on the time-sphere product, and homogeneity extends both properties
to all nonzero $u$. The first part then applies to every $Q_{t,a}$.
\end{proof}

\papersubsection{Configuration matrices and affine relations}

We represent centered six-point configurations by linear maps and
derive the rank and relation inequalities imposed by strict convexity.
These restrictions will control the singular matrices in the bordism.

Write $\mathbf1=(1,\ldots,1)\in\R^6$ and set
\begin{equation}\label{eq:simplex}
 U=\Bigl\{c\in\R^6:\sum_i c_i=0\Bigr\},\qquad
 b_i=\frac{e_i-\tfrac16\mathbf1}{\sqrt2},\quad 1\leq i\leq6.
\end{equation}
Direct calculation gives
$\langle b_i,b_j\rangle=\frac12(\delta_{ij}-\frac16)$, so the $b_i$
form a centered regular simplex of side length one. Also
$\sum_i a_ib_i=0$ holds exactly when all $a_i$ are equal.
Thus they span $U$, with sole relation $\sum_i b_i=0$.
Hence each centered tuple $y=(y_1,\ldots,y_6)$ in
$\R^5$ has a unique matrix $A_y\in\Hom(U,\R^5)$ satisfying
\begin{equation}\label{eq:matrix}
 A_yb_i=y_i.
\end{equation}
This is a linear identification of the centered configuration space
with $\Hom(U,\R^5)$. Since the centroid is zero, the affine span of
the tuple is the linear span of its vertices, namely $\operatorname{im}A_y$.
Consequently $\rank A_y$ equals its affine dimension, and
\[
 \ker A_y=\Bigl\{c\in\R^6:\sum_i c_i=0,\quad
                           \sum_i c_i y_i=0\Bigr\},
\]
since $\sum_i c_i b_i=c/\sqrt2$ for $c\in U$.

We fix an oriented isometry $U\cong\R^5$ and identify the space of
orientation-preserving isometries $U\to\R^5$ with $\SO(5)$.

The group $S_6$ acts orthogonally on $U$ by coordinate permutations.
Write $\rho$ for this representation, and use the same notation for its
restriction to $C_5=\langle(12345)\rangle$, which fixes label six.
For $\sigma\in S_6$, the relabeling action
$(\sigma\cdot y)_i=y_{\sigma^{-1}(i)}$ satisfies
$A_{\sigma\cdot y}=A_y\rho(\sigma)^{-1}$, since
$\rho(\sigma)b_i=b_{\sigma(i)}$.

For $c\in U$ define
\[
 S(c)=\frac12\sum_i|c_i|,\qquad
 \osc(c)=\max_i c_i-\min_i c_i.
\]
For a nonzero zero-sum vector, $S(c)$ is both its total positive
mass and its total negative mass in absolute value.

\begin{lemma}\label{lem:balance}
Let $y_1,\ldots,y_6$ be equilateral in a strictly convex norm. They
are strictly antipodal: for every $i\ne j$ there is a linear
functional uniquely maximal at $y_i$ and uniquely minimal at $y_j$.
Every nonzero affine relation $c$ among these points satisfies
\begin{equation}\label{eq:balance}
 \osc(c)<S(c).
\end{equation}
Moreover, distinct unordered edges have distinct directions modulo sign.
\end{lemma}

\begin{proof}
Normalize all distances to one and choose a norming functional
$\ell$ of $y_i-y_j$. Strict convexity implies that $y_i-y_j$ is
the only unit vector where $\ell$ takes the value one.
For $k\notin\{i,j\}$, the support inequality gives
$\ell(y_k-y_j)<1=\ell(y_i-y_j)$ and
$\ell(y_i-y_k)<1=\ell(y_i-y_j)$, since the corresponding unit
vectors are distinct. Rearranging gives
\[
 \ell(y_j)<\ell(y_k)<\ell(y_i).
\]
Thus the tuple is strictly antipodal. In particular, the difference
between the maximum and minimum of $\ell$ on the vertices is one.

For a nonzero affine relation $c$, choose $i$ with
$c_i=\max_k c_k>0$ and $j$ with $c_j=\min_k c_k<0$.
Normalize the corresponding affine functional to take values
$u_i=1$, $u_j=0$, and $0<u_k<1$ at every other vertex.
The relation gives $\sum_k c_ku_k=0$. Therefore
\[
 c_i\leq\sum_{c_k>0}c_ku_k
       =\sum_{c_k<0}(-c_k)u_k\leq S(c)+c_j.
\]
At least one inequality is strict: otherwise $i$ and $j$ are the
only nonzero coefficients, which would imply $y_i=y_j$.
This proves~\eqref{eq:balance}.

Finally, parallel unit edges agree after orienting them suitably.
A norming functional for their common vector is uniquely maximal and
minimal at the endpoints of each edge, as just proved. The ordered
endpoints must therefore agree.
\end{proof}

\begin{lemma}\label{lem:rank}
Every six-point equilateral configuration in a strictly convex norm
on $\R^5$ has affine rank at least four. At rank four, its unique
relation line satisfies~\eqref{eq:balance}.
\end{lemma}

\begin{proof}
By the classification theorem of Sch\"urmann and
Swanepoel~\cite[Theorem~2.3]{SchurmannSwanepoel2006}, a normed
three-space with six equilateral points has either
the one-skeleton of an affine regular octahedron on its unit sphere,
or a two-dimensional face of its unit ball. Each alternative is
incompatible with strict convexity. If the affine rank of our tuple
were at most three, translating its affine span into a
three-dimensional subspace and restricting the norm would give a
contradiction. The restricted unit sphere is the intersection of the
original unit sphere with that subspace, so it too contains no
nontrivial segment. The last assertion follows from Lemma~\ref{lem:balance}.
\end{proof}

\papersubsection{The configuration bordism}

We use the preceding perturbation and geometric lemmas to construct
a compact oriented bordism from the complete zero set of the distance
equations. We also determine its label action and its Euclidean fiber.

Let
\[
 V=\Bigl\{(y_1,\ldots,y_6)\in(\R^5)^6:\sum_i y_i=0\Bigr\},
 \qquad W=\R^{15},
\]
where the coordinates of $W$ are indexed by unordered pairs of labels.
On the open subset $V^\circ$ of tuples with distinct entries define
\begin{equation}\label{eq:residual}
 F_N(y)=\bigl(N(y_i-y_j)^2-1\bigr)_{i<j}.
\end{equation}
A smooth norm is called \emph{regular} here if $D_yF_N$ is onto
at every zero. The empty zero set is regular by convention.
The symmetric group $S_6$ acts by
$(\sigma\cdot y)_i=y_{\sigma^{-1}(i)}$ and permutes the edge
coordinates of $W$. These actions make $F_N$ equivariant.
The next lemma applies parameter transversality to the complete
zero set of~\eqref{eq:residual}.

\begin{lemma}\label{lem:bordism}
Let $N^0,N^1$ be regular smooth strongly convex norms on $\R^5$.
There is a smooth path $N^t$, constant near its endpoints, for which
\[
 \B=\bigl\{(t,y)\in[0,1]\times V^\circ:F_{N^t}(y)=0\bigr\}
\]
is a compact oriented smooth eleven-manifold, with boundary the two
endpoint fibers. The label action of $S_6$ is free. Its subgroup
$C_5=\langle(12345)\rangle$ preserves the orientation.
The Euclidean norm is regular, and its zero fiber is equivariantly
diffeomorphic to $O(U,\R^5)$.
\end{lemma}

\begin{proof}
We construct a finite-dimensional family of norm perturbations, apply
Lemma~\ref{lem:parameter}, and then identify its boundary and orientation.

\textbf{Step 1.} Construct a path with fixed endpoints and uniform bounds.

\nopagebreak[4]
Choose a smooth $\eta:[0,1]\to[0,1]$ equal to zero on $[0,1/3]$
and one on $[2/3,1]$, and set
\[
 Q_t=\bigl(1-\eta(t)\bigr)(N^0)^2+\eta(t)(N^1)^2.
\]
By Lemma~\ref{lem:perturbation}, $\sqrt{Q_t}$ is a smooth strongly
convex norm. On the compact time-sphere product, uniform equivalence
to the Euclidean norm gives constants $a_0,b_0>0$ with
$a_0|u|\leq\sqrt{Q_t(u)}\leq b_0|u|$. The same comparison, with
$a_0$ decreased and $b_0$ increased once, holds for every sufficiently
small perturbation on the sphere. At a unit-equilateral tuple,
\begin{equation}\label{eq:compact-estimates}
 b_0^{-1}\leq|y_i-y_j|\leq a_0^{-1},\qquad
 |y_i|=\Bigl|\frac16\sum_j(y_i-y_j)\Bigr|
       \leq\frac{5}{6a_0}.
\end{equation}
Consequently all zeros of all sufficiently small perturbations lie in
one fixed compact subset $D\subset[0,1]\times V^\circ$. This excludes
both escape to infinity and collisions. In particular the unperturbed
zero set is compact.

\textbf{Step 2.} Perturb the fifteen edge equations independently.

\nopagebreak[4]
Choose a smooth $\chi$ equal to one on $[1/3,2/3]$ and zero on
$[0,1/4]\cup[3/4,1]$. Fix an unperturbed zero $(t,y)$ with
$t\in[1/3,2/3]$. For each unordered edge $e=\{i,j\}$ write
$v_e=y_i-y_j$ with either choice of sign. By Lemma~\ref{lem:balance}
the fifteen points $[v_e]\in\mathbb{RP}^4$ are distinct. Choose
smooth even functions $\theta_e$ on $S^4$, each equal to one at
$\pm v_e/|v_e|$, whose supports are pairwise disjoint. Smooth bump
functions exist in these neighborhoods~\cite[Chapter~2]{Lee2013};
averaging each function with its antipodal pullback makes it even. Define
\[
 h_e(u)=|u|^2\theta_e\bigl(u/|u|\bigr),\qquad u\ne0.
\]
For the perturbation $Q_t+\chi(t)\sum_e a_eh_e$, differentiation of
the residual on edge $e'$ with respect to $a_e$ gives
\[
 \chi(t)h_e\bigl(v_{e'}\bigr)=|v_{e'}|^2\delta_{ee'}.
\]
Thus the parameter derivative is onto $W$ at this zero, and remains
onto nearby. Compactness of the zero set in the middle interval gives
finitely many such collections. List all their functions as
$h_1,\ldots,h_L$ and set
\begin{align}\label{eq:universal}
 Q_{t,a}(u)&=Q_t(u)+\chi(t)\sum_{j=1}^L a_jh_j(u),\\
 \Phi(t,y,a)&=\bigl(Q_{t,a}(y_i-y_j)-1\bigr)_{i<j}.\notag
\end{align}
The parameters belong to a small open ball $P\subset\R^L$.
For $a=0$ the total derivative of $\Phi$ is onto at every zero in
the middle interval. Outside that interval, $Q_t$ is an endpoint
squared norm and its configuration derivative is onto by endpoint
regularity. Hence $D\Phi$ is onto at every zero with $a=0$.

\textbf{Step 3.} Choose a parameter giving the required smooth bordism.

\nopagebreak[4]
Lemma~\ref{lem:perturbation} lets us shrink $P$ so that every
$N_{t,a}=\sqrt{Q_{t,a}}$ is a strongly convex norm, with the uniform
bounds used in~\eqref{eq:compact-estimates}. Shrink $P$ again so that
$\Phi$ is transverse to zero throughout $[0,1]\times V^\circ\times P$.
To justify this last choice, suppose that nontransverse zeros existed
for a sequence $a_k\to0$. Their $(t,y)$ coordinates lie in $D$, so a
subsequence converges to $(t_*,y_*)$. This is a zero at $a=0$, where
$D\Phi$ is onto by Step 2. Surjectivity is open, a contradiction.
This argument accounts for all possible new zeros of the perturbed
path, not only those in one local configuration chart.

Apply Lemma~\ref{lem:parameter} with
$M=(0,1)\times V^\circ$ and the family~\eqref{eq:universal}.
Choose a regular parameter $a\in P$ and fix $N^t=N_{t,a}$.
By Lemma~\ref{lem:regular-fiber}, the full interior zero set is a smooth
manifold of dimension $1+\dim V-\dim W=1+25-15=11$.
On the endpoint collars $\chi=0$ and $Q_t$ is constant. The zero set
there is exactly an interval times the regular endpoint fiber.
Therefore the complete zero set in $[0,1]\times V^\circ$ is a smooth
manifold with boundary exactly the two endpoint fibers. It is compact
by~\eqref{eq:compact-estimates}.

\textbf{Step 4.} Prove freeness and preservation of orientation.

\nopagebreak[4]
For any norm, relabeling a tuple simply permutes the edge residuals.
Thus every perturbation above is $S_6$-equivariant, even though no
spatial rotational invariance is imposed on the bump functions.
No nonidentity label permutation fixes a tuple of distinct points,
so the $S_6$-action on $\B$ is free.

Orient $\R\oplus V$ and $W$. Along the regular zero set the derivative
gives the equivariant exact sequence
\begin{equation}\label{eq:orientation}
 0\longrightarrow T\B\longrightarrow
 \B\times(\R\oplus V)\longrightarrow\B\times W\longrightarrow0.
\end{equation}
Writing $D$ for this surjective derivative, the Euclidean metrics give
the smooth equivariant right inverse $D^*(DD^*)^{-1}$.
Thus the sequence splits and its determinant line defines an orientation
of $\B$. The generator $\sigma=(12345)$
has determinant one on $U$: its determinant on $\R^6$ is one and it
fixes $\mathbf1$. Hence its determinant on
$V\cong\Hom(U,\R^5)$ is $(\det\rho(\sigma))^5=1$.
On $W$, the fifteen edges divide into three orbits of length five:
the five edges incident to label six, the adjacent pairs in the
five-cycle, and its nonadjacent pairs. The determinant is again one.
The time action is trivial. Taking determinant lines in
\eqref{eq:orientation} shows that the action on $T\B$ preserves
orientation. The boundary collars also give the usual induced
orientations on the endpoint fibers.

\textbf{Step 5.} Identify the Euclidean fiber and prove regularity.

\nopagebreak[4]
For a centered Euclidean unit-equilateral tuple set
$r_i=|y_i|^2$ and $R=\sum_i r_i$. Summing the five distances from
$y_i$ to the other vertices gives $6r_i+R=5$.
All $r_i$ are therefore equal, $r_i=5/12$, and
$\langle y_i,y_j\rangle=-1/12$ for $i\ne j$. This is precisely
the Gram matrix of the $b_i$ in~\eqref{eq:simplex}. The linear map
$A_yb_i=y_i$ is consequently an isometry. Conversely each isometry
gives such a tuple, and the linear dictionary~\eqref{eq:matrix}
makes this identification an equivariant diffeomorphism.

To check regularity, suppose weights $\lambda_{ij}=\lambda_{ji}$
annihilate the fifteen edge differentials on $V$. Put
$g_i=\sum_{j\ne i}\lambda_{ij}(y_i-y_j)$.
The annihilation condition is
$\sum_i\langle g_i,z_i\rangle=0$ for every centered tangent tuple
$(z_i)$. Taking $z_i=v$, $z_j=-v$ shows that all $g_i$ are equal.
Their sum is zero, so each $g_i=0$. At a fixed simplex vertex the
five incident edges form a basis of $\R^5$, and therefore every
incident weight vanishes. Thus all $\lambda_{ij}=0$, proving that
the edge differential is onto.
\end{proof}

The construction uses every zero of the distance equations.
In particular, configurations are retained when their matrices
lose rank. By Lemma~\ref{lem:rank}, the only possible ranks are
four and five.

\papersubsection{A graph model for the relation-line space}

We construct an equivariant graph model for compact regions of the
admissible relation-line space. This reduces the later homology
calculation to bundles over a finite graph.

Normalize a nonzero relation by $S(c)=1$. Define
\begin{equation}\label{eq:kernel}
 X_{\bal}=\bigl\{c\in U:S(c)=1,\ \osc(c)<1\bigr\},\qquad
 K_{\bal}=X_{\bal}/(c\sim-c).
\end{equation}
The second space parametrizes lines satisfying~\eqref{eq:balance}.
For $0<\delta<1/6$ put
\[
 X_\delta=\bigl\{c\in U:S(c)=1,\ \osc(c)\leq1-\delta\bigr\},\qquad
 K_\delta=X_\delta/(c\sim-c).
\]
These compact spaces exhaust $K_{\bal}$ as $\delta\downarrow0$.
Here $K_{\bal}$ is viewed as the open subset of $\mathbb{RP}(U)$
defined by $\osc(c)/S(c)<1$. Normalization $c\mapsto c/S(c)$
identifies its double cover with $X_{\bal}$. In particular the
topologies in~\eqref{eq:kernel} are the usual projective and quotient
topologies. Every compact subset of $K_{\bal}$ lies in some
$K_\delta$, since its positive continuous balance margin has a
positive minimum.
Let $\Gamma$ have one vertex for each subset $I\subset\{1,\ldots,6\}$
of size $2$, $3$ or $4$, and an edge between $I$ and $J$ when
$I\subset J$ and $|J|=|I|+1$. Complementation acts on $\Gamma$;
write $\bar\Gamma=\Gamma/(I\sim I^c)$.

\begin{lemma}\label{lem:graph}
For every $0<\delta<1/6$ there is an
$S_6\times\{\pm1\}$-equivariant homotopy equivalence
$X_\delta\simeq\Gamma$, where $-1$ acts by complementation on
$\Gamma$. Consequently $K_\delta\simeq\bar\Gamma$
$C_5$-equivariantly. The $C_5$-actions on $K_{\bal}$ and
$\bar\Gamma$ are free. The connected graph $\bar\Gamma/C_5$ has
five vertices and twelve edges.
\end{lemma}

\begin{proof}
We divide the proof into three steps.

\textbf{Step 1.} Determine the sign-cell intersections.

\nopagebreak[4]
If $c\in X_\delta$ has $p$ positive and $q$ negative coordinates,
then
\[
 \osc(c)\geq\frac1p+\frac1q.
\]
If $p=1$ or $q=1$, the right side exceeds one. If $p+q\leq4$,
then $1/p+1/q\geq4/(p+q)\geq1$. Both possibilities contradict
$\osc(c)<1$. Hence $p,q\geq2$ and $p+q\geq5$, so at most one
coordinate vanishes. For each allowed sign set $I$ define the
compact convex polytope
\[
 C_I=\left\{c\in U:
 \begin{aligned}
 &\sum_{i\in I}c_i=1,\quad c_i\geq0,\ i\in I,\quad c_j\leq0,\ j\notin I,\\
 &c_i-c_j\leq1-\delta,\quad i\in I,\ j\notin I
 \end{aligned}\right\}.
\]
The defining conditions are linear equalities and inequalities,
and $|c_i|\leq1$, so $C_I$ is indeed compact and convex.
These polytopes cover $X_\delta$, by assigning a possible zero
coordinate to either sign. If $C_I$ and $C_J$ meet, all
coordinates in $I\mathbin{\triangle}J$ vanish. Hence distinct cells
meet only for an edge of $\Gamma$, and no three meet. Conversely,
if $I\subset J$ and $|J|=|I|+1$, their intersection contains the
point with coordinate zero on $J\setminus I$, value $1/|I|$ on
$I$, and value $-1/(5-|I|)$ outside $J$. Its oscillation is
$5/6<1-\delta$.

\textbf{Step 2.} Construct equivariant maps in both directions.

\nopagebreak[4]
We give an explicit equivariant homotopy equivalence. Map the vertex
$I$ to the point $b_I$ with coordinates $1/|I|$ on $I$ and
$-1/(6-|I|)$ on $I^c$. Map the midpoint of the edge $I-J$ to
the point $b_{IJ}$ just described, and map each half-edge linearly.
The resulting map $f:\Gamma\to X_\delta$ is equivariant. Its
half-edge images lie in the corresponding convex cells. Indeed the
oscillation of $b_I$ is $3/4$ for $|I|=2,4$ and $2/3$ for $|I|=3$,
while that of $b_{IJ}$ is $5/6$; all are at most $1-\delta$.
Convexity keeps every joining segment in the appropriate cell.

For an edge $e=I-J$ write $E_e=C_I\cap C_J$. The sets $E_e$
are pairwise disjoint and compact. Choose $\eta>0$ so small that
their $\eta$-neighborhoods are disjoint, miss all nonincident cells
and all points $b_I$, and, for each $e$, miss every half-edge image
other than the two belonging to $e$.
Such a choice is possible: a half-edge image meets the shared
faces only at its own midpoint, as its other coordinates retain
their strict signs. The distance and the finite collection of
conditions are invariant under permutations and sign reversal.

For $c\in C_I$ outside these neighborhoods put $g(c)=I$.
If $c$ is within distance $\eta$ of $E_e$ for an incident edge,
put $g(c)$ on that half-edge at fraction
$1-\operatorname{dist}(c,E_e)/\eta$ of the way from $I$ to the
midpoint. There is at most one such edge. On $E_e$ the two
definitions agree at the midpoint. This defines a continuous
equivariant map $g:X_\delta\to\Gamma$. More explicitly, its restriction
to each $C_I$ is continuous: at distance $\eta$ the half-edge parameter
tends to zero, and on a shared face it tends to the midpoint from
either cell. Since the cells form a finite closed cover and the
cellwise maps agree on their intersections, they paste to a continuous
map on $X_\delta$.

\textbf{Step 3.} Prove the homotopy equivalence and freeness.

\nopagebreak[4]
Both $c$ and $f(g(c))$ lie in the same cell, and in the same shared
face when $c$ lies on one. The formula
$H(c,s)=(1-s)c+sf(g(c))$ is therefore a well-defined homotopy in
$X_\delta$. This point matters because $X_\delta$ itself need not
be convex: interpolation is performed within a single cell, where
the normalization and sign conditions remain valid. On each
half-edge, the chosen neighborhood separations ensure that $gf$
stays in that half-edge and fixes its endpoint and midpoint.
Linear interpolation in this interval gives $gf\simeq\mathrm{id}$.
The homotopies agree at vertices and midpoints and commute with
permutations and sign reversal.

Complementation fixes neither a vertex nor an edge of $\Gamma$.
Thus the equivalence descends to $K_\delta\simeq\bar\Gamma$.
There are $15+20+15=50$ vertices and $15\cdot4+20\cdot3=120$
edges before complementing, so $\bar\Gamma$ has $25$ vertices
and $60$ edges. The graph is connected: two-element subsets can
be joined by exchanging one element through a three-element
superset, and the remaining vertices join to these.

No nonidentity element of $C_5$ fixes a vertex of $\bar\Gamma$.
An invariant subset is a union of the five-element orbit and the
singleton orbit, so its size is $0$, $1$, $5$ or $6$. If an
element carried $I$ to $I^c$, its square would fix $I$, giving
the same contradiction. An odd-order action cannot interchange
the endpoints of a fixed edge. The action on the graph is
therefore free. Likewise, a fixed real projective line for a
nontrivial five-cycle has a representative $c$ with
$\rho(\sigma)^rc=\alpha c$. Since $\alpha$ is real and
$\alpha^5=1$, it equals one. The first five coordinates are equal
and the zero-sum condition forces the sixth to be minus five times
their value. Thus the representative is proportional to
$(1,1,1,1,1,-5)$. This vector has $\osc(c)/S(c)=6/5$, so the
line is outside $K_{\bal}$. Dividing the graph counts by five
finishes the proof.
\end{proof}

\papersection{Proof of the main results}
\label{sec:proof}

We prove Theorem~\ref{thm:main} by combining the configuration bordism
with the graph model from Section~\ref{sec:configurations}. First we show
that the boundary associated with the relation graph contributes zero
to tenth homology. We then cut at a regular determinant level and
compare this contribution with the nonzero Euclidean boundary class.

\papersubsection{Homology of bundles over the relation graph}

We compute the homology groups needed for the graph model and show
that its associated bundle maps trivially to the quotient isometry
group in degree ten.

The group $\SO(5)$ is a compact connected Lie group of dimension ten
\cite[Section~3.D]{Hatcher2002}.

\begin{lemma}\label{lem:homology}
One has $H_9(\SO(5);\F_5)=0$. If $E\to T$ is a fiber bundle over
a finite graph with fiber $\SO(5)$, then $H_{10}(E;\F_5)$ is
generated by the images of the fundamental classes of fibers over
the vertices of $T$.
\end{lemma}

\begin{proof}
The standard quaternionic covering~\cite[Section~3.D, p.~294]{Hatcher2002} is
\[
 S^3\times S^3\longrightarrow\SO(4),\qquad
 (a,b)\longmapsto\bigl(x\longmapsto axb^{-1}\bigr).
\]
Its kernel is $\{(1,1),(-1,-1)\}$. Indeed, if $axb^{-1}=x$ for
all quaternions $x$, then $a=b$ by taking $x=1$, and $a$ commutes
with all quaternions, so $a=b=\pm1$.
At the identity its differential sends a pair of imaginary quaternions
$(\alpha,\beta)$ to the skew-adjoint map $x\mapsto\alpha x-x\beta$.
If this map is zero, then $\alpha=\beta$ is both imaginary and central,
so it is zero. The differential is therefore an isomorphism between
six-dimensional Lie algebras. Its image is open, and is closed by
compactness of $S^3\times S^3$. Since $\SO(4)$ is connected, the
map is surjective and is the asserted two-sheeted covering.
The nontrivial deck transformation is translation by $(-1,-1)$,
which is homotopic to the identity since $S^3\times S^3$ is connected.
Formula~\eqref{eq:transfer} says that both transfer composites are
multiplication by two. Therefore the covering induces a homology
isomorphism over $\F_5$. Applying~\eqref{eq:kunneth} to $S^3\times S^3$
gives $H_5(\SO(4);\F_5)=0$. This six-manifold has no homology above
degree six~\cite[Theorem~3.26]{Hatcher2002}.

The last-column map $\SO(5)\to S^4$ is a locally trivial bundle
with fiber $\SO(4)$~\cite[Example~4.55]{Hatcher2002}.
Cover $S^4$ by two slightly enlarged open hemispheres $U_+$ and $U_-$.
The bundle is trivial over each: stereographic coordinates from the
opposite pole give a tangent frame; Gram--Schmidt orthonormalization
and an orientation choice, followed by adjoining the unit normal,
give a smooth section of the frame bundle. Their intersection retracts to $S^3$, and a
trivialization over $U_+$ restricts to one over that intersection.
Consequently the inverse images have homotopy type $\SO(4)$ and
their intersection has homotopy type $\SO(4)\times S^3$. The relevant
Mayer--Vietoris terms from~\eqref{eq:mv-sequence} are
\[
 0=H_9\bigl(\SO(4);\F_5\bigr)^{\oplus2}\longrightarrow
 H_9\bigl(\SO(5);\F_5\bigr)\longrightarrow H_8\bigl(\SO(4)\times S^3;\F_5\bigr)=0,
\]
where~\eqref{eq:sphere-product} gives
\[
 H_8\bigl(\SO(4)\times S^3;\F_5\bigr)
 =H_8\bigl(\SO(4);\F_5\bigr)\oplus H_5\bigl(\SO(4);\F_5\bigr)=0.
\]
This proves the first assertion.

For the second, let $E_0$ be the union of fibers over the vertices.
Pull the bundle back to each characteristic interval of an edge and
trivialize it there. This description also applies to loop edges: their
two endpoint fibers are attached to the appropriate vertex fibers by
possibly different identifications. Short open stars of the vertices
lift to a neighborhood of $E_0$ which deformation retracts onto $E_0$:
in each incident edge chart, contract the interval coordinate towards
its endpoint. These contractions fix the endpoint fibers and hence
agree with their gluing maps. Thus $(E,E_0)$ is a good pair, as required in
Lemma~\ref{lem:products-quotients}.
After collapsing $E_0$, the quotient is a wedge, one summand per edge,
of the spaces
\[
 \bigl(\SO(5)\times[0,1]\bigr)/\bigl(\SO(5)\times\{0,1\}\bigr).
\]
The image of the two endpoint collars in each summand is a
neighborhood that contracts to its basepoint, so the wedge also
satisfies the good-pair hypothesis. Formulas~\eqref{eq:good-quotient}
and~\eqref{eq:wedge-homology}, followed by~\eqref{eq:interval-product}, give
\begin{align*}
H_{10}(E,E_0;\F_5)
 &\cong\bigoplus_{\text{edges}}
 H_{10}\bigl(\SO(5)\times[0,1],\SO(5)\times\{0,1\};\F_5\bigr)\\
 &\cong\bigoplus_{\text{edges}}H_9\bigl(\SO(5);\F_5\bigr)=0.
\end{align*}
The exact sequence~\eqref{eq:pair-sequence} shows that
$H_{10}(E_0;\F_5)\to H_{10}(E;\F_5)$ is onto. This argument allows
nontrivial monodromy and does not require the bundle over $T$ to be
trivial.
\end{proof}

Restrict the representation $\rho$ to $C_5$. A five-cycle is an even
permutation, so $\rho(C_5)\subset\SO(U)$. We use the actions
\[
 A\longmapsto A\rho(\sigma)^{-1},\qquad
 L\longmapsto\rho(\sigma)L
\]
on matrices and relation lines. Set $M=\SO(5)/C_5$.
Right translations preserve orientation, so $M$ is a connected
closed oriented ten-manifold and $[M]\ne0$ in $H_{10}(M;\F_5)$.

\begin{lemma}\label{lem:vanishing}
For $0<\delta<1/6$, let
\[
 E_\delta=\bigl(\SO(5)\times K_\delta\bigr)/C_5,
 \qquad p_\delta:E_\delta\to M,\quad [V,L]\longmapsto[V],
\]
where $C_5$ acts diagonally by
$(V,L)\mapsto(V\rho(\sigma)^{-1},\rho(\sigma)L)$.
Then
\[
 (p_\delta)_*:H_{10}(E_\delta;\F_5)\longrightarrow H_{10}(M;\F_5)
 \quad\text{is zero}.
\]
\end{lemma}

\begin{proof}
The equivariant maps and homotopies in Lemma~\ref{lem:graph}, applied
only to the second factor, give a homotopy equivalence over $M$
between $E_\delta$ and
\[
 E_\Gamma=\bigl(\SO(5)\times\bar\Gamma\bigr)/C_5.
\]
The latter projects to $\bar\Gamma/C_5$ as a bundle with fiber
$\SO(5)$, because the action on $\bar\Gamma$ is free. To see the
local trivialization explicitly, take an evenly covered neighborhood
$J$ in $\bar\Gamma/C_5$ and a lift $s:J\to\bar\Gamma$.
Every orbit over $J$ has a unique representative $(V,s(x))$,
giving $\SO(5)\times J$ as a product chart. Different lifts change
$V$ by a right translation which is locally constant on the base
(constant on each connected overlap). Thus this is a bundle even if
the covering of the graph is nontrivial.
Lemma~\ref{lem:homology} says that its tenth homology is generated
by vertex-fiber classes. On each such fiber, the map to $M$ is
the five-sheeted covering $\SO(5)\to\SO(5)/C_5$. Its degree is
five, which is zero in $\F_5$. The homotopy equivalence from
$E_\delta$ to $E_\Gamma$ preserves the first-coordinate projection
to $M$ throughout both homotopies, so this vanishing is precisely
the asserted vanishing of $(p_\delta)_*$.
\end{proof}

\papersubsection{The determinant cut}

We use the preceding homological vanishing to rule out a bordism with
the rank and relation constraints established in Section~\ref{sec:configurations}.
Applying this obstruction to the configuration bordism proves
Theorem~\ref{thm:main}.

\begin{lemma}\label{lem:cut}
There is no compact oriented smooth eleven-manifold $B$ with the
following data and properties: a free smooth action of $C_5$ preserving
orientation, and a smooth equivariant map $A:B\to\Hom(U,\R^5)$ satisfying
\begin{enumerate}[label=\textup{(\roman*)}]
\item the entire boundary is equivariantly identified with
      $O(U,\R^5)$, and $A$ restricts there to the inclusion;
\item $\rank A(b)\geq4$ for all $b\in B$;
\item $\ker A(b)\in K_{\bal}$ whenever $\rank A(b)=4$.
\end{enumerate}
\end{lemma}

\begin{proof}
We divide the proof into three steps.

\textbf{Step 1.} Define the polar map and prove rank loss.

\nopagebreak[4]
Suppose such $B$ and $A$ exist, and set $d(b)=\det A(b)$.
This is $C_5$-invariant and equals $+1$ and $-1$ on the two
boundary components. For a positive-determinant matrix define
its positive polar factor
\begin{equation}\label{eq:polar}
 V(A)=A\bigl(A^*A\bigr)^{-1/2}\in\SO(5).
\end{equation}
It is smooth. For $R=\rho(\sigma)$, we have
\[
 \bigl(AR^{-1}\bigr)^*\bigl(AR^{-1}\bigr)=R\bigl(A^*A\bigr)R^{-1},\qquad
 V\bigl(AR^{-1}\bigr)=V(A)R^{-1}.
\]
Indeed, the inverse positive square root of the first expression is
$R(A^*A)^{-1/2}R^{-1}$, which gives the second identity and proves
equivariance.
We use $V$ only where $d$ is strictly positive.

If $d$ never vanishes, the part $\{d>0\}$ is compact and has
boundary $\SO(5)$. Its quotient maps to $M$ by~\eqref{eq:polar},
restricting to the identity on the boundary. Formula~\eqref{eq:boundary}
would give $[M]=0$, a contradiction. Hence $D_0=d^{-1}(0)$ is
nonempty.

\textbf{Step 2.} Extend the kernel line with a uniform balance margin.

\nopagebreak[4]
Order the singular values as $s_1\geq\cdots\geq s_5$.
Compactness and the rank bound give $s_4\geq a>0$ on $B$.
Since $|d|=s_1s_2s_3s_4s_5$, on $|d|\leq t$ we have
$s_5\leq t/a^4$. Choose $t<a^5/2$. Then $s_5<a/2$ while
$s_4\geq a$, so the least eigenvalue of $A^*A$ is simple and
separated by a uniform gap from the others. Its eigenspace $L(b)$
is a continuous line on this neighborhood of $D_0$. One way to
check continuity is to use its spectral projection
\[
 \Pi(b)=\frac{1}{2\pi\mathrm i}
       \int_{|z|=a^2/2}\bigl(zI-A(b)^*A(b)\bigr)^{-1}\,dz.
\]
The contour is traversed counterclockwise and the resolvent is taken
on the complexification of $U$. Diagonalizing the real symmetric
matrix shows that the contour encloses only the least eigenvalue
and that $\Pi$ is its real orthogonal projection. It has rank one
and depends smoothly on $b$ on the open spectral-gap region.
Conjugating $A^*A$ by $\rho(\sigma)$ conjugates $\Pi$ by the
same matrix, so $L$ is equivariant. On $D_0$, it is exactly $\ker A$.

For a line $L=\R c\subset U$ put
\[
 \mu(L)=1-\frac{\osc(c)}{S(c)}.
\]
This is continuous and independent of the nonzero representative.
It is positive on $L(D_0)$ by assumption. Let $m_0>0$ be its minimum
there. Shrink $t$ so that $\mu(L(b))\geq m_0/2$ whenever
$0\leq d(b)\leq t$, and choose $0<\delta<\min(1/6,m_0/2)$.
Then
\begin{equation}\label{eq:margin}
 L(b)\in K_\delta,\qquad 0\leq d(b)\leq t.
\end{equation}
Indeed, failure for every smaller $t$ would give a sequence
$b_k$ with $d(b_k)\downarrow0$ and $\mu(L(b_k))<m_0/2$.
A subsequence converges in $B$ to a point of $D_0$, where continuity
would contradict the defining minimum $m_0$.

\textbf{Step 3.} Cut at a regular level and compare boundary classes.

\nopagebreak[4]
Choose a regular value $0<\varepsilon<\min(t,1)$ of $d$ and set
\[
 B_\varepsilon=\bigl\{b\in B:d(b)\geq\varepsilon\bigr\},\qquad
 Z_\varepsilon=d^{-1}(\varepsilon),\qquad
 Y_\varepsilon=B_\varepsilon/C_5.
\]
Apply Lemma~\ref{lem:sard} to $d$ on the interior of $B$ to obtain
such a value. Since $d=\pm1$ on the original boundary, the level is
disjoint from that boundary, so $Y_\varepsilon$ is a compact oriented
manifold with boundary $M\sqcup(Z_\varepsilon/C_5)$.
There are no corners: near the original boundary, $d=\pm1$ on
that boundary and continuity keeps the new level away from it;
at the new level, regularity gives an ordinary half-space chart.
Lemma~\ref{lem:finite-quotient} gives the quotient charts, orientations
and boundary identification.
Polar decomposition induces $P:Y_\varepsilon\to M$ which is the
identity on $M$. On the other boundary its restriction factors as
\begin{align*}
 Z_\varepsilon/C_5&\longrightarrow E_\delta
       \xrightarrow{\ p_\delta\ }M,\\
 [b]&\longmapsto\Bigl[V\bigl(A(b)\bigr),L(b)\Bigr]
       \longmapsto\Bigl[V\bigl(A(b)\bigr)\Bigr].
\end{align*}
This is well defined by~\eqref{eq:margin} and equivariance.
Lemma~\ref{lem:vanishing} makes its map on tenth homology zero.
The level may be empty or disconnected; an empty level contributes
zero, and a nonempty compact level has finitely many components.
This holds for every component of $Z_\varepsilon/C_5$, with its
induced orientation, and therefore for their sum. On the connected
Euclidean component, the induced orientation is either the standard
orientation of $M$ or its negative, so the identity map contributes
exactly $\pm[M]$. Applying~\eqref{eq:boundary} to $P$ gives
\[
 0=P_*[\partial Y_\varepsilon]=\pm[M]+0,
\]
contrary to $[M]\ne0$.
\end{proof}

\begin{proof}[Proof of Theorem~\ref{thm:main}]
Suppose a norm on $\R^5$ has no six-point equilateral set.
By Lemma~\ref{lem:smoothing}, choose a smooth strongly convex
norm $N^1$ with the same property. Let $N^0$ be Euclidean.
Both endpoints are regular, the second because its fiber is empty.
Lemma~\ref{lem:bordism} gives a compact oriented
eleven-manifold $\B$ whose entire boundary is $O(U,\R^5)$.
The $C_5$-action is free and orientation preserving.

Assign to each $(t,y)\in\B$ the matrix $A_y$ from~\eqref{eq:matrix}.
This map is smooth, equals the boundary isometry at the Euclidean
endpoint, and satisfies
\[
 A_{\sigma\cdot y}=A_y\rho(\sigma)^{-1}.
\]
All norms on the path are strictly convex. Lemma~\ref{lem:rank}
therefore gives $\rank A_y\geq4$, and every rank-four kernel lies
in $K_{\bal}$. This contradicts Lemma~\ref{lem:cut}.
Thus the six points exist. Translating one of them to zero and
dividing by their common distance gives the stated normalization.
\end{proof}

\begin{remark}\label{rem:limitation}
For $n\geq6$, the rank bound and graph reduction above give no
corresponding control of $(n+1)$-point configurations. With seven labels,
$c=(1/3,1/3,1/3,-1/2,-1/2,0,0)$ satisfies $S(c)=1$ and
$\osc(c)=5/6$, but four sign cells meet at its two zero coordinates.
Thus the pairwise-intersection construction of Lemma~\ref{lem:graph}
fails; appending zeros gives the same example with more labels.
Without suitable rank and relation-space control, Lemma~\ref{lem:homology}
does not reduce the determinant-level homology to fiber classes, so
the covering-degree cancellation is unavailable. This limitation
does not decide the conjecture in the remaining dimensions.
\end{remark}

\par\vskip30pt
% Bibliography conventions follow arXiv:2608.13536v2.

\clearpage
\appendix
\noindent\begin{minipage}{\textwidth}
\papersection{Prompt provided to hyprover}\label{app:prompt}

The following prompt was written by the author and provided to hyprover.

\medskip
\noindent\textbf{Five-dimensional Petty problem}

\medskip
\noindent\textbf{Objective}

Prove that every five-dimensional real normed vector space contains six equilateral points.
Precisely, for every norm \(\|\cdot\|\) on \(\mathbb{R}^5\), prove that there exist distinct points
\(x_0,\ldots,x_5\) such that
\[
\|x_i-x_j\|=1 \qquad (0\leq i<j\leq 5).
\]
The norm is arbitrary; affine independence of the six points is not required.

\medskip
\noindent\textbf{Research hint}

Seek an appropriate differential-topological reformulation while preserving the original geometric constraints.
Differential geometry and algebraic topology may help; consult Lu--Yang's
\href{https://arxiv.org/abs/2608.13536}{\emph{A Solution to Banach's Isometric Conjecture}}
for methodological inspiration, not a ready-made argument.

\medskip
\noindent\textbf{Working requirements}

\begin{itemize}[label=\textendash,leftmargin=1.5em]
\item Work in English and preserve the exact target throughout.
\item Supply a complete, independently checkable proof, including all essential intermediate arguments and precise references for external results.
\item Distinguish proved statements, conjectures, failed approaches, and unresolved gaps. Never replace a missing argument with a claim of verification.
\item Do not count special cases, conditional reductions, numerical evidence, or a proof sketch as a solution.
\end{itemize}
\end{minipage}

\end{document}